\documentclass[a4paper,11pt]{amsart}
\usepackage{mathrsfs}
\usepackage{cases}

\usepackage{color}
\usepackage{colordvi}
\usepackage{amsfonts}

\usepackage{graphicx}
\usepackage{amsmath}
\usepackage{amssymb}
\usepackage[all]{xypic}
\usepackage[all]{xy}
\begin{document}
\input xy
\xyoption{all}

\renewcommand{\mod}{\operatorname{mod}\nolimits}
\newcommand{\proj}{\operatorname{proj}\nolimits}
\newcommand{\rad}{\operatorname{rad}\nolimits}
\newcommand{\Gproj}{\operatorname{Gproj}\nolimits}
\newcommand{\Ginj}{\operatorname{Ginj}\nolimits}
\newcommand{\Gd}{\operatorname{Gd}\nolimits}
\newcommand{\soc}{\operatorname{soc}\nolimits}
\newcommand{\ind}{\operatorname{inj.dim}\nolimits}
\newcommand{\Top}{\operatorname{top}\nolimits}
\newcommand{\ann}{\operatorname{Ann}\nolimits}
\newcommand{\id}{\operatorname{id}\nolimits}
\newcommand{\Id}{\operatorname{id}\nolimits}
\newcommand{\Mod}{\operatorname{Mod}\nolimits}
\newcommand{\End}{\operatorname{End}\nolimits}
\newcommand{\Ob}{\operatorname{Ob}\nolimits}
\newcommand{\Ht}{\operatorname{Ht}\nolimits}
\newcommand{\cone}{\operatorname{cone}\nolimits}
\newcommand{\rep}{\operatorname{rep}\nolimits}
\newcommand{\Ext}{\operatorname{Ext}\nolimits}
\newcommand{\Hom}{\operatorname{Hom}\nolimits}
\newcommand{\RHom}{\operatorname{RHom}\nolimits}
\renewcommand{\Im}{\operatorname{Im}\nolimits}
\newcommand{\Ker}{\operatorname{Ker}\nolimits}
\newcommand{\Coker}{\operatorname{coker}\nolimits}
\renewcommand{\dim}{\operatorname{dim}\nolimits}
\newcommand{\Ab}{{\operatorname{Ab}\nolimits}}
\newcommand{\Coim}{{\operatorname{Coim}\nolimits}}
\newcommand{\pd}{\operatorname{proj.dim}\nolimits}
\newcommand{\sdim}{\operatorname{sdim}\nolimits}
\newcommand{\add}{\operatorname{add}\nolimits}
\newcommand{\pr}{\operatorname{pr}\nolimits}
\newcommand{\Tr}{\operatorname{Tr}\nolimits}
\newcommand{\Def}{\operatorname{Def}\nolimits}
\newcommand{\Gp}{\operatorname{Gproj}\nolimits}

\newcommand{\ca}{{\mathcal A}}
\newcommand{\cb}{{\mathcal B}}
\newcommand{\cc}{{\mathcal C}}
\newcommand{\cd}{{\mathcal D}}
\newcommand{\cg}{{\mathcal G}}
\newcommand{\cp}{{\mathcal P}}
\newcommand{\ce}{{\mathcal E}}
\newcommand{\cs}{{\mathcal S}}
\newcommand{\cm}{{\mathcal M}}
\newcommand{\cn}{{\mathcal N}}
\newcommand{\cx}{{\mathcal X}}
\newcommand{\ct}{{\mathcal T}}
\newcommand{\cu}{{\mathcal U}}
\newcommand{\co}{{\mathcal O}}
\newcommand{\cv}{{\mathcal V}}
\newcommand{\calr}{{\mathcal R}}
\newcommand{\ol}{\overline}
\newcommand{\ul}{\underline}
\newcommand{\st}{[1]}
\newcommand{\ow}{\widetilde}
\newcommand{\coh}{{\mathrm coh}}
\newcommand{\CM}{{\mathrm CM}}
\newcommand{\vect}{{\mathrm vect}}

\def \A{{\Bbb A}}
\newcommand{\bp}{{\mathbf p}}
\newcommand{\bL}{{\mathbf L}}
\newcommand{\bS}{{\mathbf S}}

\newtheorem{theorem}{Theorem}[section]
\newtheorem{acknowledgement}[theorem]{Acknowledgement}
\newtheorem{algorithm}[theorem]{Algorithm}
\newtheorem{axiom}[theorem]{Axiom}
\newtheorem{case}[theorem]{Case}
\newtheorem{claim}[theorem]{Claim}
\newtheorem{conclusion}[theorem]{Conclusion}
\newtheorem{condition}[theorem]{Condition}
\newtheorem{conjecture}[theorem]{Conjecture}
\newtheorem{construction}[theorem]{Construction}
\newtheorem{corollary}[theorem]{Corollary}
\newtheorem{criterion}[theorem]{Criterion}
\newtheorem{definition}[theorem]{Definition}
\newtheorem{example}[theorem]{Example}
\newtheorem{exercise}[theorem]{Exercise}
\newtheorem{lemma}[theorem]{Lemma}
\newtheorem{notation}[theorem]{Notation}
\newtheorem{problem}[theorem]{Problem}
\newtheorem{proposition}[theorem]{Proposition}
\newtheorem{remark}[theorem]{Remark}
\newtheorem{solution}[theorem]{Solution}
\newtheorem{summary}[theorem]{Summary}
\newtheorem*{thm}{Theorem}
\newtheorem*{thma}{Theorem A}
\newtheorem*{thmb}{Theorem B}
\newtheorem*{thmc}{Theorem C}
\newtheorem*{thm1}{Main Theorem 1}
\newtheorem*{thm2}{Main Theorem 2}
\def \Z{{\Bbb Z}}
\def \X{{\Bbb X}}
\renewcommand{\L}{{\Bbb L}}
\renewcommand{\P}{{\Bbb P}}
\newcommand {\lu}[1]{\textcolor{red}{$\clubsuit$: #1}}
\title[Gorenstein defect categories of triangular matrix algebras]{Gorenstein defect categories of triangular matrix algebras}

\author[Lu]{Ming Lu}
\address{Department of Mathematics, Sichuan University, Chengdu 610064, P.R.China}
\email{luming@scu.edu.cn}

\subjclass[2000]{18E30, 18E35}
\keywords{Gorenstein defect category, triangular matrix algebra, Recollement, Gorenstein algebra}

\begin{abstract}
We apply the technique of recollement to study the Gorenstein defect categories of triangular matrix algebras. First, we construct a left recollement of Gorenstein defect
categories for a triangular matrix algebra under some conditions, using it, we give a categorical interpretation of
the Gorenstein properties of the triangular matrix algebra obtained by X-W. Chen, B. L. Xiong and P. Zhang respectively. Second, under some additional conditions, a recollement of Gorenstein defect
categories for a triangular matrix algebra is constructed. As an application, for a special kind of triangular matrix algebras, which are called simple gluing algebras, we describe their singularity categories and Gorenstein defect categories.
\end{abstract}

\maketitle

\section{Introduction}

In the study of B-branes on Landau-Ginzburg models in the framework of Homological Mirror Symmetry Conjecture, D. Orlov rediscovered the notion of singularity categories \cite{Or1,Or2,Or3}. The singularity category of an algebra $A$ is defined to be the Verdier quotient of the bounded derived category with respect to the thick subcategory formed by complexes isomorphic to those consisting of finitely generated projective modules, \cite{Bu}. It measures the homological singularity of an algebra in the sense that an algebra $A$ has finite global dimension if and only if its singularity category vanishes.

The singularity category captures the stable homological features of an algebra \cite{Bu}. A fundamental result of R. Buchweitz \cite{Bu} and D. Happel \cite{Ha1} states that for a Gorenstein algebra $A$, the singularity category is triangle equivalent to the stable category of Gorenstein projective (also called (maximal) Cohen-Macaulay) $A$-modules. Buchweitz's Theorem (\cite[Theorem 4.4.1]{Bu}) says that there is an exact embedding $\Phi:\underline{\Gp}A\rightarrow D_{sg}(A)$ given by $\Phi(M)=M$, where the second $M$ is the corresponding stalk complex at degree $0$, and $\Phi$ is an equivalence if and only if $A$ is Gorenstein. Recently, to provide a categorical characterization of Gorenstein algebras, P. A. Bergh, D. A. J{\o}rgensen and S. Oppermann \cite{BJO} defined the Gorenstein defect category $D_{def}(A):=D_{sg}(A)/\Im \Phi$ and proved that $A$ is Gorenstein if and only if $D_{def}(A)=0$. In general, it is difficult to describe the singularity categories and Gorenstein defect categories. Many people are trying to describe these categories for some special kinds of algebras, see e.g. \cite{Chen1,chen2,chen3,Ka,CGLu,CDZ}. In particular, for a CM-finite algebra $A$, F. Kong and P. Zhang \cite{KZ} proved that its Gorenstein defect category is equivalent to the singularity category of its Cohen-Macaulay Auslander algebra. Recently, X-W. Chen \cite{chen3} described the singularity category and Gorenstein defect category for a quadratic monomial algebra. For a triangular matrix algebra $\Lambda=\left( \begin{array}{cc} A&M\\0&B \end{array}\right)$ with bimodule $_AM_B$, X-W. Chen \cite{Chen1}, B. L. Xiong and P. Zhang \cite{XZ} obtained sufficient and necessary conditions for $\Lambda$ to be a Gorenstein algebra.

In this paper, we mainly consider the Gorenstein defect category of a triangular matrix algebra by using the technique of recollement.
A recollement of triangulated categories is a diagram
\[\xymatrix{
\mathcal{D}'\ar[r]^{i_*}&\mathcal{D}\ar@<2ex>[l]^{i^!}\ar@<-3ex>[l]_{i^*}\ar[r]^{j^*}&\mathcal{D}''
\ar@<2ex>[l]^{j_*}\ar@<-3ex>[l]_{j_!}
}
\]of triangulated categories and functors satisfying various conditions, which describes the middle term as being ``glued together" from a triangulated subcategory and another one. The recollement setup was first introduced by A. Beilinson, J. Bernstein and P. Deligne in \cite{BBD}, which plays an important role in  algebraic geometry and representation theory, see for instance \cite{MV, CPS1, CPS2, J, Ko, Zhang2}. Another interesting structure is the left recollement, which is a part of recollement involving only the functors $i^*,i_*,j_!,j^*$, see e.g. \cite{Zhang2}.

It is well known that for a triangular matrix algebra $\Lambda=\left( \begin{array}{cc} A&M\\0&B \end{array}\right)$, $D^b(\Lambda)$ admits a recollement relative to $D^b(A)$ and $D^b(B)$ if $\pd M_B<\infty$, see e.g. \cite{CL}.
P. Zhang proved that if $\Lambda$ is Gorenstein and $_AM$ is projective, then $\underline{\Gp}\Lambda$ admits a recollement relative to $\underline{\Gp}A$ and $\underline{\Gp}B$. The author together with P. Liu \cite{LL} generalized this to consider the singularity categories, proved that $D_{sg}(\Lambda)$ admits a recollement relative to $D_{sg}(A)$ and $D_{sg}(B)$ if $\pd_A M<\infty$ and $\pd M_B<\infty$ for any Artin rings $A$ and $B$.
In \cite{LL}, the machinery of localization to the three categories in a recollement is introduced, and some sufficient conditions for the quotient categories to form a new recollement are found. We refer the reader to \cite{Kra} for localization theory of triangulated categories.

The aim of the present article is to study the Gorenstein defect categories for triangular matrix algebra $\Lambda$. Let $\Lambda=\left( \begin{array}{cc} A&M\\0&B \end{array}\right)$ be a triangular matrix algebra with $\pd M_B<\infty$. First, if $\pd _AM<\infty$ or $\ind _A M<\infty$, then we have a left recollement of $D_{def}(\Lambda)$ relative to $D_{def}(A)$ and $D_{def}(B)$, see Proposition \ref{proposition left recollement}. If additionally $A$ (resp. $B$) is Gorenstein, then $D_{def}(\Lambda)$ is triangle equivalent to $D_{def}(B)$ (resp. $D_{def}(A)$), see Theorem \ref{theorem characterize of gorenstein property}. This gives a categorical interpretation of the suffiecient and necessary conditions for $\Lambda$ to be a Gorenstein algebras in \cite{Chen1,XZ}, see Corollary \ref{corollary triangular matrix algebra to be Gorenstein}. Second, we construct a recollement for the Gorenstein defect category $D_{def}(\Lambda)$ under some conditions, explicitly, 
if $_AM$ is projective and $M_B$ has finite projective dimension, and $\Hom_A(M,A) \in (\Gproj B)^\bot$, or in particular either $\pd_B\Hom_A(M,P)<\infty$ or $\ind_B\Hom_A(M,P)<\infty$ for every indecomposable projective $A$-module $P$, then $D_{def}(\Lambda)$ admits a recollement relative to $D_{def}(A)$ and $D_{def}(B)$, see Theorem \ref{theorem recollement of Gorenstein defect categories}. As a corollary, $D_{def}(\left( \begin{array}{cc} A&A\\0&A \end{array}\right))$ admits a recollement relative to $D_{def}(A)$ and $D_{def}(A)$, see Corollary \ref{corollary special triangle matrix algebra}. Finally, as an application, we consider a special kind of triangular matrix algebras, which are called simple gluing algebras, and get that
$D_{sg}(\Lambda)\simeq D_{sg}(A)\coprod D_{sg}(B)$ and $D_{def}(\Lambda)\simeq D_{def}(A)\coprod D_{def}(B)$, see Proposition \ref{lemma recollement splits} and Theorem \ref{theorem defect Gorenstein split}.

\vspace{0.2cm} \noindent{\bf Acknowledgments.}
The author thanks Professor Liangang Peng very much for his guidance and constant support. The author thanks the referee for very helpful and insightful comments. The author was supported by the National Natural Science Foundation of China (No. 11401401 and No. 11601441).

\section{preliminary}
In this paper, we always assume that $R$ is a commutative Artinian ring and all algebras are Artin  $R$-algebras and modules are considered to be finitely generated. We denote by $\mod A$ the category of finitely generated left $A$-modules. Right $A$-modules are viewed as left $A^{op}$-modules, here $A^{op}$ denotes the opposite algebra of $A$. In what follows, $A$-modules always mean left $A$-modules.

We denote by $D:\mod A \rightarrow \mod A^{op}$ the duality functor $\Hom_R(-,E)$, where $E$ is the minimal injective cogenerator for $R$. Note that $D$ is isomorphic to $\Hom_A(-,D(A_A))$, and $D(A_A)$ is an injective cogenerator for $\mod A$. For an arbitrary $A$-module $_AX$ we denote by $\pd_AX$ (resp. $\ind_AX$) the projective dimension (resp. the injective dimension) of the module $_AX$.

\subsection{Recollements}
We recall the definition of recollement for triangulated categories from \cite{BBD}.

\begin{definition}[\cite{BBD}]
A recollement of triangulated categories is a diagram of triangulated
categories and triangulated functors
\[
\xymatrix{
\mathcal{D}'\ar[r]^{i_*=i_!}&\mathcal{D}\ar@<2ex>[l]^{i^!}\ar@<-3ex>[l]_{i^*}\ar[r]^{j^*=j^!}&\mathcal{D}''
\ar@<2ex>[l]^{j_*}\ar@<-3ex>[l]_{j_!} \quad(\dag)
}
\]
satisfying
\item[(R1)] $(i^*,i_*)$, $(i_*,i^!)$, $(j_!,j^*)$ and $(j^*,j_*)$ are adjoint pairs;
\item[(R2)] $i_*$, $j_!$ and $j_*$ are full embeddings;
\item[(R3)] $j^*i_*=0$ (and hence, by adjoint properties, $i^*j_!=0$ and $i^!j_*=0$);
\item[(R4)] Each object $X$ in $\cd$ determines distinguished triangles
\[i_!i^!X\to X\to j_*j^*X\to \cdot\ \text{and}\ j_!j^!X\to X\to i_*i^*X\to\cdot
\]in $\cd$, where the arrows to and from $X$ are counit and unit morphisms.
\end{definition}

Parallel to the recollement of triangulated categories, \emph{recollement of abelian categories} can be dated back to \cite{BBD} in the construction of the category of perverse sheaves on a singular space, we refer to \cite{FP} for its definition and basic properties. Let $\cd,\cd',\cd''$ be abelian categories. The diagram $(\dag)$ of additive functors is an abelian category recollement of $\cd$ relative to $\cd'$ and $\cd''$, if $(R1),(R2)$ and $(R5)$ are satisfied, where
\begin{itemize}
\item[(R5)] $\Im i_*=\ker j^*$.
\end{itemize}

The following theorem is useful to construct recollements of quotient categories.
\begin{theorem}[\cite{LL}]\label{theorem 1}
Let \[\xymatrix{
\mathcal{D}'\ar[r]^{i_*}&\mathcal{D}\ar@<2ex>[l]^{i^!}\ar@<-3ex>[l]_{i^*}\ar[r]^{j^*}&\mathcal{D}''
\ar@<2ex>[l]^{j_*}\ar@<-3ex>[l]_{j_!}
}
\]be a recollement of triangulated categories and $\ct$ be a thick subcategory of $\cd$. Let $\mathcal{T}_1=i^*\mathcal{T}$ and $\mathcal{T}_2=j^*\mathcal{T}$.  If $i_*i^*\mathcal{T}\subseteq \mathcal{T}$ and $j_*j^*\mathcal{T}\subseteq \mathcal{T}$, then there exists a recollement of localizations
\[\xymatrix{\mathcal{D}'/\mathcal{T}_1\ar[r]^{\widetilde{i}_*}&\mathcal{D}/\mathcal{T}
\ar@<2ex>[l]^{\widetilde{i}^!}\ar@<-3ex>[l]_{\widetilde{i}^*}\ar[r]^{\widetilde{j}^*}&\mathcal{D}''/\mathcal{T}_2.
\ar@<2ex>[l]^{\widetilde{j}_*}\ar@<-3ex>[l]_{\widetilde{j}_!}  }
\]
\end{theorem}

A \emph{left recollement} of a triangulated category $\cd$ relative to triangulated categories $\cd'$ and $\cd''$ is a diagram of exact functors consisting of the upper two rows of $(\dag)$, satisfying all conditions which involve only the functors $i^*,i_*,j_!,j^*$.

Similar to the proof of Theorem \ref{theorem 1} in \cite{LL}, we get the following result about the left recollement of localizations.

\begin{lemma}\label{theorem left recollement}
Let \[\xymatrix{
\mathcal{D}'\ar[r]^{i_*}&\mathcal{D}\ar@<-3ex>[l]_{i^*}\ar[r]^{j^*}&\mathcal{D}''
\ar@<-3ex>[l]_{j_!}
}
\]be a left recollement of triangulated categories and $\ct$ be a thick subcategory of $\cd$. Let $\mathcal{T}_1=i^*\mathcal{T}$ and $\mathcal{T}_2=j^*\mathcal{T}$.  If $i_*i^*\mathcal{T}\subseteq \mathcal{T}$, then there exists a left recollement of localizations
\[\xymatrix{\mathcal{D}'/\mathcal{T}_1\ar[r]^{\widetilde{i}_*}&\mathcal{D}/\mathcal{T}
\ar@<-3ex>[l]_{\widetilde{i}^*}\ar[r]^{\widetilde{j}^*}&\mathcal{D}''/\mathcal{T}_2.
\ar@<-3ex>[l]_{\widetilde{j}_!}  }
\]
\end{lemma}

\subsection{Gorenstein projective modules and Gorenstein algebras}
Let $A$ be an Artin algebra.
A complex $$P^\bullet:\cdots\rightarrow P^{-1}\rightarrow P^0\xrightarrow{d^0}P^1\rightarrow \cdots$$ of finitely generated projective $A$-modules is said to be \emph{totally acyclic} provided it is acyclic and the Hom complex $\Hom_A(P^\bullet,A)$ is also acyclic \cite{AM}.
An $A$-module $M$ is said to be (finitely generated) \emph{Gorenstein projective} provided that there is a totally acyclic complex $P^\bullet$ of projective $A$-modules such that $M\cong \Ker d^0$ \cite{EJ}. We denote by $\Gproj A$ the full subcategory of $\mod A$ consisting of Gorenstein projective modules.

Let $\cx$ be a subcategory of $\mod A$. Then $^\bot\cx:=\{M|\Ext_A^i(M,X)=0, \mbox{ for all } X\in\cx, i\geq1\}$. Dually, we can define $\cx^\bot$. In particular, we define $^\bot A:=^\bot
(\proj A)$.

The following lemma follows directly from the definition of Gorenstein projective modules.
\begin{lemma}\label{lemma property of Gorenstein projective modules} Let $A$ be an Artin algebra. Then

(a) \cite{Be}
\begin{eqnarray*}
\Gp(A)&=&\{M\in \mod A |\mbox{there is an exact sequence }
0\rightarrow M\rightarrow T^0\xrightarrow{d^0}T^1\xrightarrow{d^1}\cdots, \\ &&\mbox{ with }T^i\in\proj A,\ker d^i\in\,^\bot A,\forall i\geq0\}.
\end{eqnarray*}
In particular, $\ker d^i\in \Gp(A)$ for each $i\geq0$.

(b) If $M$ is Gorenstein projective, then $\Ext^i_A(M,L)=0$, $\forall i>0$, for all $L$ of finite projective dimension or of finite injective dimension.

(c) If $P^\bullet$ is a totally acyclic complex, then all $\Im d^i$ are Gorenstein projective; and any
truncations
$$\cdots\rightarrow P^i\rightarrow\Im d^i\rightarrow0,\quad 0\rightarrow\Im d^i\rightarrow P^{i+1}\rightarrow\cdots$$
and
$$0\rightarrow\Im d^i\rightarrow P^{i+1}\rightarrow\cdots\rightarrow P^j\rightarrow \Im d^j\rightarrow0,i<j$$
are $\Hom_A(-,\proj A)$-exact.
\end{lemma}

For a module $X$, take a short exact sequence $$0\rightarrow \Omega X\rightarrow P\rightarrow X\rightarrow0$$
with $P$ projective. The module $\Omega X$ is called a \emph{syzygy module} of $X$. Obviously, if $X$ is Gorenstein projective, then so is $\Omega X$.

Let $\ct$ be the smallest subcategory of $\mod A$ containing all modules of finite projective dimension or of finite injective dimension and closed under direct summands and extensions.
From Lemma \ref{lemma property of Gorenstein projective modules} (b), it is easy to see that $\ct\subseteq (\Gproj A)^\bot$.

\begin{definition}[\cite{Ha1}, see also \cite{AR1,AR2}]
An Artin algebra $A$ is called a Gorenstein algebra (or Iwanaga-Gorenstein algebra) if $A$ satisfies $\ind A_A<\infty$ and $\ind_AA<\infty$.
Given an $A$-module $X$. If $\Ext^i_A(X,A)=0$ for all $i>0$, then $X$ is called a (maximal) Cohen-Macaulay module of $A$.
\end{definition}

Observe that for a Gorenstein algebra $A$, we have $\ind _AA=\ind A_A$, see \cite[Lemma 6.9]{Ha1}; the common value is denoted by $\Gd A$. If $\Gd A\leq d$, we say that $A$ is \emph{$d$-Gorenstein}.

\begin{theorem}[\cite{Bu,EJ}]
Let $A$ be a Gorenstein algebra. Then

(a) If $P^\bullet$ is an exact sequence of projective left $A$-modules, then $\Hom_A(P^\bullet,A)$ is again an exact sequence of projective right $A$-modules.

(b) A module $G$ is Gorenstein projective if and only if there is an exact sequence $0\rightarrow G\rightarrow P^0\rightarrow P^1\rightarrow \cdots$ with each
$P^i$ projective.

(c) $\Gproj A=\,^\bot A$.
\end{theorem}

So for a Gorenstein algebra, the definition of Cohen-Macaulay module coincides with the one of Gorenstein projective.

Recall that for an algebra $A$, its \emph{singularity category} is the quotient category $D_{sg}(A):=D^b(A)/K^b(\proj A)$, which is defined by Buchweitz \cite{Bu}, see also \cite{Ha1,Or1}.

\begin{theorem}[Buchweitz's Theorem, see also \cite{KV} for a more general version]\label{theorem stable category of CM modules }
Let $A$ be an Artin algebra. Then $\Gproj (A)$ is a Frobenius category with the projective modules as the projective-injective objects, and there is an exact embedding $\Phi:\underline{\Gp}A\rightarrow D_{sg}(A)$ given by $\Phi(M)=M$, where the second $M$ is the corresponding stalk complex at degree $0$, and $\Phi$ is an equivalence if and only if $A$ is Gorenstein.
\end{theorem}

Let $A$ be an Artin algebra. Inspired by Buchweitz's Theorem, the \emph{Gorenstein defect category} is defined to be Verdier quotient $D_{def}(A):=D_{sg}(A)/\Im(\Phi)$, see \cite{BJO}.
From \cite{KZ}, we know that $D_{def}(A)$ is triangle equivalent to $D^b(A)/\langle \Gp(A)\rangle$, where $\langle \Gp(A)\rangle$ denotes the triangulated subcategory of $D^b(A)$ generated by $\Gp(A)$, i.e., the smallest triangulated subcategory of $D^b(A)$ containing $\Gp(A)$.

\begin{lemma}[\cite{BJO,KZ}]
Let $A$ be an Artin algebra. Then the following are equivalent.

(a) $A$ is Gorenstein;

(b) $\underline{\Gp}(A)$ is triangle equivalent to $D_{sg}(A)$;

(c) $D_{def}(A)=0$;

(d) $D^b(A)=\langle \Gp(A)\rangle$.
\end{lemma}

\section{Gorenstein defect categories}

Let $A$ and $B$ be two Artin algebras, $_AM_B$ an $A\mbox{-}B$-bimodule, and $\Lambda=\left( \begin{array}{cc} A&M\\0&B \end{array}\right)$.
Let $G=M\otimes_B -$ and $F$ be its right adjoint functor $\Hom_A(M,-)$. We denote by $\alpha_{X,Y}$ the adjoint isomorphism
$$\alpha_{X,Y}: \Hom_A(G(Y),X)\rightarrow \Hom_B(Y,F(X)),$$
and
$\psi_X=\alpha^{-1}_{X,F(X)}(\Id_{F(X)})$ for any $X\in\mod A$.

A left $\Lambda$-module is identified with a triple $\left( \begin{array}{cc} X\\Y \end{array}\right)_\phi$, where $X\in \mod A$, $Y\in \mod B$, and $\phi:M\otimes_BY\rightarrow X$ is a morphism of $A$-modules. A morphism $\left(\begin{array}{cc} X\\ Y \end{array}\right)_\phi\rightarrow \left(\begin{array}{cc} X' \\ Y' \end{array}\right)_{\phi'}$
is a pair $\left(\begin{array}{cc} f\\ g \end{array}\right)$, where $f\in \Hom_A(X,X')$, $g\in\Hom_B(Y,Y')$, such that $\phi'(\id_M\otimes g)=f\phi$.
So $\ker(\left(\begin{array}{cc} f\\ g \end{array}\right) )= \left(\begin{array}{cc} \ker (f)\\ \ker(g) \end{array}\right)_\psi$
for some induced (unique) morphism $\psi: G(\ker(g))\rightarrow \ker (f)$. In particular, $\left(\begin{array}{cc} f\\ g \end{array}\right)$ is injective if and only if $f$ and $g$ are injective. Similarly, we get that
$\Im(\left(\begin{array}{cc} f\\ g \end{array}\right) )= \left(\begin{array}{cc} \Im(f)\\ \Im(g) \end{array}\right)_{\psi'}$
for some induced (unique) morphism $\psi': G(\Im(g))\rightarrow \Im (f)$ since $G$ is right exact, and
$\left(\begin{array}{cc} f\\ g \end{array}\right)$ is surjective if and only if $f$ and $g$ are surjective.

Note that $\left(\begin{array}{cc} P \\ 0 \end{array}\right)$ and $\left(\begin{array}{cc} G(Q) \\ Q \end{array}\right)_{\id}$ are precisely the indecomposable projective modules, where $P$ and $Q$ are indecomposable projective as $A$-module and $B$-module respectively.
Dually, $\left(\begin{array}{cc} J \\ F(J) \end{array}\right)_{\psi_J}$ and $\left(\begin{array}{cc} 0 \\ I \end{array}\right)_{\id}$ are precisely the indecomposable injective modules, where $J$ and $I$ are indecomposable injective as $A$-module and $B$-module respectively.
We refer the reader to \cite{P,Zhang2} for the statements here.

\begin{theorem}[\cite{P,Zhang2}]\label{recollement of abelian categories}
Let $A$ and $B$ be Artin algebras, $_AM_B$ an $A\mbox{-}B$-bimodule, and $\Lambda=\left( \begin{array}{cc} A&M\\0&B \end{array}\right)$. Denote by $G=M\otimes_B-$. We have the following recollement of abelian categories:
\[\xymatrix{
\mod A \ar[r]^{i_*}& \mod \Lambda
\ar@<2ex>[l]^{i^!}\ar@<-3ex>[l]_{i^*}\ar[r]^{j^*}& \mod B
\ar@<2ex>[l]^{j_*}\ar@<-3ex>[l]_{j_!} , }\]
where $i^*$ is given by $\left(\begin{array}{cc} X  \\ Y   \end{array} \right)_{\phi }\mapsto \Coker(\phi )$;
$i_*$ is given by $X \mapsto \left(\begin{array}{cc} X  \\ 0  \end{array} \right)$; $i^!$ is given by $\left(\begin{array}{cc} X  \\ Y   \end{array} \right)_{\phi }\mapsto X $;
$j_!$ is given by $Y  \mapsto\left(\begin{array}{cc} G(Y ) \\ Y   \end{array} \right)_{\Id}$;
$j^*$ is given by $\left(\begin{array}{cc} X  \\ Y   \end{array} \right)_{\phi }\mapsto Y $;
$j_*$ is given by $ Y \mapsto \left(\begin{array}{cc} 0 \\ Y   \end{array} \right)$.
\end{theorem}

Note that the functors $i_*,i^!,j^*,j_*$ defined above are exact and the functors $i^*,i_*,j_!,j^*$ preserve projective modules.

\begin{lemma}\label{lemma adjoint pair}
Keep the notations as above. Then

(a) $i^!$ admits a right adjoint functor $i_?:\mod A\rightarrow \mod \Lambda$ given by $X\mapsto \left( \begin{array}{cc} X\\ F(X) \end{array} \right)_{\psi_X}$, where $\psi_X=\alpha^{-1}_{X,F(X)}(\Id_{F(X)})$;

(b) $j_*$ admits a right adjoint functor $j_?: \mod \Lambda\rightarrow \mod B$
given by $\left( \begin{array}{cc} X\\ Y \end{array} \right)_{\phi}\mapsto \ker (\alpha_{X,Y}(\phi))$.
\end{lemma}
\begin{proof}
(a) Let $\left(\begin{array}{cc} X  \\ Y   \end{array} \right)_{\phi }\in\mod \Lambda$ and $X'\in\mod A$. For any morphism $f$ in $\Hom_{A}(i^!(\left(\begin{array}{cc} X  \\ Y   \end{array} \right)_{\phi }) ,X')=\Hom_A(X,X')$, by the naturality of the adjoint pair $(G,F)$, we get the following commutative diagram
\[\xymatrix{ G(Y) \ar[rrr]^{ G(F(f)\alpha_{X,Y} (\phi))} \ar[d]^{\phi}&&&  GF(X') \ar[d]^{\psi_{X'}} \\
X\ar[rrr]^{f} &&& X'. }\]
So $\left(\begin{array}{cc} f  \\   F(f)\alpha_{X,Y} (\phi)  \end{array} \right): \left(\begin{array}{cc} X  \\ Y   \end{array} \right)_{\phi }\rightarrow \left( \begin{array}{cc} X'\\ F(X') \end{array} \right)_{\psi_{X'}}$ is a $\Lambda$-morphism.
Define $$\beta_{X', \left(\begin{array}{cc} X  \\ Y   \end{array} \right)_{\phi }}: \Hom_{A}(i^!(\left(\begin{array}{cc} X  \\ Y   \end{array} \right)_{\phi }) ,X')\rightarrow \Hom_\Lambda( \left(\begin{array}{cc} X  \\ Y   \end{array} \right)_{\phi } , i_?(X') )$$
by mapping $f$ to $\left(\begin{array}{cc} f  \\   F(f)\alpha_{X,Y} (\phi)  \end{array} \right)$ for any $f\in \Hom_{A}(i^!(\left(\begin{array}{cc} X  \\ Y   \end{array} \right)_{\phi }) ,X')=\Hom_A(X,X')$. Obviously, it is well-defined, and injective.
Conversely, it is easy to see that any morphism in $\Hom_\Lambda( \left(\begin{array}{cc} X  \\ Y   \end{array} \right)_{\phi } , i_?(X') )$ is of form
$\left(\begin{array}{cc} f  \\   F(f)\alpha_{X,Y} (\phi)  \end{array} \right)$ for some (unique) morphism $f:X\rightarrow X'$. So
$\beta_{X', \left(\begin{array}{cc} X  \\ Y   \end{array} \right)_{\phi }}$ is surjective.
From the naturality of the adjoint pair $(G,F)$, it is routine to check that this isomoprhism $\beta_{X', \left(\begin{array}{cc} X  \\ Y   \end{array} \right)_{\phi }}$ is natural, and then $(i^!,i_?)$ is an adjoint pair.

The proof of (b) is similar to that of (a), we omit it here.
\end{proof}

\begin{lemma}[see e.g. \cite{chen2}]\label{lemma functor to derived functor}
Let $F_1:\ca\rightarrow\cb$ be an exact functor between abelian categories which has an exact right adjoint $F_2$. Then the pair $(D^b(F_1),D^b(F_2))$ is adjoint, where $D^b(F_1)$ is the induced functor from $D^b(\ca)$ to $D^b(\cb)$ ($D^b(F_2)$ is defined similarly). Moreover, if $F_1$ is fully faithful, then so is $D^b(F_1)$.
\end{lemma}

The following well-known result is very helpful.

\begin{theorem}[see e.g. \cite{CL}]\label{recollement of derived categories}
Let $A$ and $B$ be two Artin algebras, $_AM_B$ an $A\mbox{-}B$-bimodule and $\Lambda=\left( \begin{array}{cc} A&M\\0&B \end{array}\right)$.
If $M$ as a right module has finite projective dimension, then we have the following recollement:
\[\xymatrix{
D^b(A) \ar[r]^{D^b(i_*)}& D^b(\Lambda)
\ar@<2ex>[l]^{D^b(i^!)}\ar@<-3ex>[l]_{\L(i^*)}\ar[r]^{D^b(j^*)}& D^b(B)
\ar@<2ex>[l]^{D^b(j_*)}\ar@<-3ex>[l]_{\L(j_!)}  }\]
where $\L(i^*), D^b(i_*),D^b(i^!),\L(j_!),D^b(j^*),D^b(j_*)$ are the derived functors of these in Theorem \ref{recollement of abelian categories}.
\end{theorem}

Note that the recollement in Theorem \ref{recollement of derived categories} is exactly the derived version of the recollement in Theorem \ref{recollement of abelian categories}. For convenience, we also use $i^*,i_*,i^!,j_!,j^*,j_*$ to denote their derived functors respectively if there is no confusion.

\begin{lemma}[\cite{Zhang2}]\label{lemma zhang}
Let $A$ and $B$ be two Artin algebras, $_AM_B$ an $A\mbox{-}B$-bimodule with $\pd M_B<\infty$ and $\Lambda=\left( \begin{array}{cc} A&M\\0&B \end{array}\right)$. If $\pd_AM<\infty$ or $\ind_A M<\infty$, then the $\Lambda$-module $\left(\begin{array}{cc} X\\ Y   \end{array}  \right)_{\phi}$ is Gorenstein projective if and only if $Y$ is Gorenstein projective, $\phi$ is injective, and $\Coker(\phi)$ is Gorenstein projective.
\end{lemma}

\begin{proposition}\label{proposition left recollement}
Let $A$ and $B$ be two Artin algebras, $_AM_B$ an $A\mbox{-}B$-bimodule with $\pd M_B<\infty$ and $\Lambda=\left( \begin{array}{cc} A&M\\0&B \end{array}\right)$.
If $\pd _AM<\infty$ or $\ind _A M<\infty$, then we have the following left recollement:
\[\xymatrix{
D_{def}(A) \ar[r]^{\tilde{i}_*}& D_{def}(\Lambda)
\ar@<-3ex>[l]_{\tilde{i}^*}\ar[r]^{\tilde{j}^*}& D_{def}(B)
\ar@<-3ex>[l]_{\tilde{j}_!}  }\]
where $\tilde{i}^*$, $\tilde{i}_*$, $\tilde{j}_!$ and $\tilde{j}^*$ are induced by the six structure functors in Theorem \ref{recollement of derived categories}.
\end{proposition}
\begin{proof}
We apply Lemma \ref{theorem left recollement} to prove it.

First, in combination with the definition of the functors in Theorem \ref{recollement of abelian categories}, from Lemma \ref{lemma zhang}, we get that $i_*(\Gproj (A))\subseteq \Gproj(\Lambda)$, $i^*(\Gproj(\Lambda))\subseteq \Gproj(A)$, $j^*(\Gproj(\Lambda))\subseteq \Gproj(B)$ and $j_!(\Gproj(B))\subseteq \Gproj(\Lambda)$. In particular, $i_*i^* (\Gproj(\Lambda))\subseteq \Gproj(\Lambda)$.  Since $i^*i_*\simeq \Id$, we get that
$i^*(\Gproj(\Lambda))\simeq \Gproj(A)$; since $j^*j_!\simeq \Id$, we get that $j^*(\Gproj(\Lambda))\simeq \Gproj(B)$.

Since $i_*$ and $j^*$ are exact functors, it is easy to see that $D^b(i_*)(\langle \Gproj(A)\rangle)\subseteq \langle\Gproj(\Lambda)\rangle$ and $D^b(j^*)(\langle \Gproj(\Lambda))\simeq \langle \Gproj (B)\rangle$.

We claim that the restriction of $i^*$ to $\Gproj(\Lambda)$ and the restriction of $j_!$ to $\Gproj(B)$ are exact.
For any exact sequence $0\rightarrow Y_1\rightarrow Y_2\rightarrow Y_3\rightarrow0$ in $\Gp(B)$, since $\pd M_B<\infty$ we get that $\ind_B D(M)<\infty$ and then $D \mathrm{Tor}^B_1(M,Y_3)= \Ext^1_B(Y_3,D(M))\simeq0$, so $0\rightarrow G(Y_1)\rightarrow G(Y_2) \rightarrow G(Y_3)\rightarrow0$ is exact, which implies that the restriction of $G$ to $\Gp(B)$ is exact. From the definition of $j_!$, it is easy to see that the restriction of $j_!$ to $\Gproj(B)$ is exact, which implies that for any Gorenstein projective $B$-module $Y$, $\L(j_!)(Y)\cong j_!(Y)$ in $D^b(\Lambda)$. So $\L(j_!)(\Gproj(B))=j_!(\Gproj(B))\subseteq \Gproj(\Lambda)$.
Then $\L(j_!)(\langle \Gproj (B)\rangle)\subseteq \langle \Gproj (\Lambda)\rangle$ since $\L(j_!)(\Gproj(B))\subseteq \Gproj(\Lambda)$.

Lemma \ref{lemma zhang} shows that for any Gorenstein projective $\Lambda$-module $\left(\begin{array}{cc} X\\ Y   \end{array}  \right)_{\phi}$, $\phi:G(Y)\rightarrow X$ is injective, which implies that the restriction of $i^*$ to $\Gproj(\Lambda)$ is exact. Similar to the above, we get that $\L(i^*)(\langle \Gproj (\Lambda)\rangle)\simeq \langle \Gproj (A)\rangle$ since $i^* (\Gproj(\Lambda))\simeq \Gproj(A)$. Together with $D^b(i_*)(\langle \Gproj(A)\rangle)\subseteq \langle\Gproj(\Lambda)\rangle$, we get that
$D^b(i_*)\L(i^*)( \langle \Gproj (\Lambda)\rangle)\subseteq \langle \Gproj (\Lambda)\rangle$.

By the definition of Gorenstein defect category, Lemma \ref{theorem left recollement} yields that there is a left recollement of Gorenstein defect categories
\[\xymatrix{
D_{def}(A) \ar[r]^{\tilde{i}_*}& D_{def}(\Lambda)
\ar@<-3ex>[l]_{\tilde{i}^*}\ar[r]^{\tilde{j}^*}& D_{def}(B).
\ar@<-3ex>[l]_{\tilde{j}_!}  }\]
\end{proof}

\begin{example}
If neither $\pd _AM<\infty$ nor $\ind _A M<\infty$, then Proposition \ref{proposition left recollement} is not true in general.

Let $K$ be a field, $Q_\Lambda$ be the quiver $\xymatrix{1 \ar@/^/[r]^{\alpha} & 2 \ar@/^/[l]^\beta & 3 \ar[l]_\gamma   }$, and $\Lambda=KQ_\Lambda/I_\Lambda$, where $I_\Lambda$ is generated by
$\alpha\beta,\beta\alpha,\beta\gamma$. Let $e_i$ be the idempotent corresponding to the vertex $i$. Set $e=e_1+e_2$ and $A=e\Lambda e$. Then
$\Lambda= \left( \begin{array}{cc} A&M\\0&B \end{array}\right)$ with $B=K$, so $\pd M_B=0$.
It is easy to see that neither $\pd _AM<\infty$ nor $\ind _A M<\infty$.

From \cite{chen3}, we know that $\Lambda$ is CM-free, i.e. $\Gp(\Lambda)=\proj \Lambda$, then $D_{def}(\Lambda)=D_{sg}(\Lambda)$ which is not zero since $\Lambda$ is not Gorenstein. However, $A$ and $B$ are self-injective, so $D_{def}(A)=D_{def}(B)=0$. Therefore, there is not any left recollement of $D_{def}(\Lambda)$ relative to $D_{def}(A)$ and $D_{def}(B)$.
\end{example}

\begin{theorem}\label{theorem characterize of gorenstein property}
Let $A$, $B$ be two Artin algebras, $M$ an $A\mbox{-}B$-bimodule and $\Lambda=\left( \begin{array}{cc} A&M\\0&B \end{array}\right)$. If $\pd M_B<\infty$ and $\pd_A M<\infty$ or $\ind_A M<\infty$, then

(a) If $B$ is a Gorenstein algebra, then $D_{def}(\Lambda)\simeq D_{def}(A)$.

(b) If $A$ is a Gorenstein algebra, then $D_{def}(\Lambda)\simeq D_{def}(B)$.
\end{theorem}
\begin{proof}
Proposition \ref{proposition left recollement} implies that $\tilde{i}_*$ and $\tilde{j}_!$ are full embeddings, and for any $Z\in D_{def}(\Lambda)$, there is a triangle in $D_{def}(\Lambda)$:
$$\tilde{j}_! \tilde{j}^* Z \rightarrow Z\rightarrow \tilde{i}_*\tilde{i}^* Z  \rightarrow \tilde{j}_! \tilde{j}^* Z[1].$$

We only need to prove (a) since (b) is similar.

If $B$ is a Gorenstein algebra, then $D_{def}(B)=0$, and so $\tilde{j}^* Z=0$ for any $Z\in D_{def}(\Lambda)$. In combination with the above triangle, we get that $Z\cong \tilde{i}_*\tilde{i}^* Z$ in $D_{def}(\Lambda)$, and then $Z\in\Im \tilde{i}_*$, which means that $\tilde{i}_*$ is dense. Therefore, $\tilde{i}_*: D_{def}(A)\rightarrow D_{def}(\Lambda)$ is an equivalence.
\end{proof}

\begin{corollary}[\cite{Chen1,XZ}]\label{corollary triangular matrix algebra to be Gorenstein}
Let $A$, $B$ be two Artin algebras, $M$ an $A\mbox{-}B$-bimodule and $\Lambda=\left( \begin{array}{cc} A&M\\0&B \end{array}\right)$. If $\pd M_B<\infty$ and $\pd_A M<\infty$ or $\ind_A M<\infty$, then $\Lambda$ is Gorenstein if and only if $A$ and $B$ are Gorenstein.
\end{corollary}
\begin{proof}
If $\Lambda$ is Gorenstein, then $D_{def}(\Lambda)=0$. Since $\tilde{i}_*$ and $\tilde{j}_!$ are full embeddings, we get that $D_{def}(A)=0=D_{def}(B)$, which implies that $A$ and $B$ are Gorenstein algebras.

If $A$ and $B$ are Gorenstein, then it follows from Theorem \ref{theorem characterize of gorenstein property} that $D_{def}(\Lambda)=0$ and so $\Lambda$ is Gorenstein.
\end{proof}

Let $A_1$ and $A_2$ be Artin algebras, and $F_1:\mod A_1\rightarrow \mod A_2$ be an additive functor. For any $X\in\mod A_1$, we call that $X$ \emph{has the property $(*)$ relative to $F_1$} if $X$ satisfies that $\cdots \xrightarrow{F_1(d_2)} F_1(P_1)\xrightarrow{F_1(d_1)} F_1(P_0) \rightarrow F_1(X)\rightarrow0$ is a projective resolution of
$F_1(X)$ for any projective resolution $\cdots \xrightarrow{d_2} P_1\xrightarrow{d_1} P_0 \rightarrow X\rightarrow0$ of $X$.

\begin{lemma}\label{lemma extension adjoint functor}
Let $A_1$ and $A_2$ be Artin algebras, and $F_1:\mod A_1\rightarrow \mod A_2$ be a functor with the property that it preserves projective objects and admits a right adjoint functor $F_2$.
For any $X\in\mod A_1$, if $X$ has the property $(*)$ relative to $F_1$, then $\Ext_{A_2}^k(F_1(X),Y) \cong \Ext_{A_1}^k (X,F_2(Y))$ for any $Y\in \mod A_2$ and $k\geq1$.
\end{lemma}
\begin{proof}
We denote by $X_i$ the kernel of $d_i$ for any $i>0$. Then $0\rightarrow X_1\rightarrow P_0\rightarrow X\rightarrow0$ is a short exact sequence.
For any $Y\in\mod A_2$, we have a long exact sequence
$$0\rightarrow \Hom_{A_1}( X,F_2(Y)) \rightarrow \Hom_{A_1}( P_0,F_2(Y)) \rightarrow \Hom_{A_1}( X_1,F_2(Y))\rightarrow \Ext_{A_1}^1( X,F_2(Y))\rightarrow0.$$

From the assumption, we get that $0\rightarrow F_1(X_1)\rightarrow F_1(P_0)\rightarrow F_1(X)\rightarrow0$ is also exact with $F_1(P_0)$ projective. So
 we have a long exact sequence
 $$0\rightarrow \Hom_{A_2}( F_1(X),Y) \rightarrow \Hom_{A_2}( F_1(P_0),Y) \rightarrow \Hom_{A_2}( F_1(X_1),Y)\rightarrow \Ext_{A_2}^1( F_1(X),Y)\rightarrow0.$$
From the naturality of the adjoint pair $(F_1,F_2)$ we get that $\Ext_{A_2}^1(F_1(X),Y) \cong \Ext_{A_1}^1 (X,F_2(Y))$ for any $Y\in \mod A_2$.

For $k>1$, it is easy to see that $\Ext^k_{A_1}(X, F_2(Y)) \cong \Ext^{k-1}_{A_1}(X_1,F_2(Y))$, and $\Ext^k_{A_2}(F_1(X),Y) \cong \Ext^{k-1}_{A_2}(F_1(X_1),Y)$.
Since $X_i$ has the property $(*)$ relative to $F_1$ for each $i$, inductively, we get
$$\Ext^{k-1}_{A_1}(X_1,F_2(Y))\cong \Ext^{k-1}_{A_2}(F_1(X_1),Y),$$ which implies that
$\Ext^k_{A_1}(X, F_2(Y))\cong \Ext^k_{A_2}(F_1(X),Y)$.
\end{proof}

\begin{lemma}\label{lemma adjoint functor Gorenstein projective}
Let $A_1$ and $A_2$ be Artin algebras, and $F_1:\mod A_1\rightarrow \mod A_2$ be a functor with the property that it preserves projective objects and admits a right adjoint functor $F_2$.

(a) If $F_1$ satisfies the following two conditions,
 \begin{itemize}
\item[(i)] $\Ext^k_{A_1} (X,F_2(Q))=0$ for any projective $A_2$-module $Q$, any Gorenstein projective $A_1$-module $X$ and any $k>0$;
 \item[(ii)] for any $X_1\in \Gp(A_1)$ and short exact sequence $0\rightarrow X_2\rightarrow P\rightarrow X_1\rightarrow0$ in $\mod A_1$ with $P$ projective, we have that $0\rightarrow F_1(X_2)\rightarrow F_1(P)\rightarrow F_1(X_1)\rightarrow0$ is also exact,
\end{itemize}
then $F_1(\Gp(A_1))\subseteq \Gp(A_2)$.

(b) If $F_1$ is exact or its restriction to $\Gp(A_1)$ is exact, and $\pd F_2(Q)<\infty$ or $\ind F_2(Q)<\infty$ for any indecomposable projective $A_2$-module $Q$, then $F_1(\Gp(A_1)) \subseteq \Gp(A_2)$.
\end{lemma}
\begin{proof}
(a) For any $X\in \Gp(A_1)$, let $0\rightarrow X_2\rightarrow P\rightarrow X\rightarrow0$ be a short exact sequence with $P$ projective. Then $X_2$ is a Gorenstein projective module. From the assumption, we have that $0\rightarrow F_1(X_2)\rightarrow F_1(P)\rightarrow F_1(X)\rightarrow0$ is also exact and $F_1(P)$ is projective.
Then $X$ satisfies the property $(*)$ relative to $F_1$.
So for any projective $A_2$-module $Q$ and Gorenstein $A_1$-module $X$, Lemma \ref{lemma extension adjoint functor} shows that $\Ext^k_{A_2}(F_1(X),Q)\cong \Ext^k_{A_1}(X,F_2(Q))=0$ for any $k>0$.
Since $X$ is Gorenstein projective, Lemma \ref{lemma property of Gorenstein projective modules} (a) shows that there is an exact sequence
$$0\rightarrow X\rightarrow P_0 \xrightarrow{d_0} P_1 \xrightarrow{d_1} P_2\xrightarrow{d_2}\cdots$$
with $P_i\in\proj A_1$, $\ker d_i \in \,^\bot A_1$, for any $i\geq0$. In particular, $\ker d_i\in \Gp(A_1)$. From above, we get that $F_1(\ker d_i) \in \,^\bot  A_2$.
By the assumption,
$$0\rightarrow F_1(X)\rightarrow F_1(P_0) \xrightarrow{F_1(d_0)} F_1(P_1) \xrightarrow{F_1(d_1)} F_1(P_2) \xrightarrow{F_1(d_2)} \cdots $$ is exact with
$\ker (F_1(d_i)) =F_1(\ker d_i) \in \, ^\bot A_2$. So $F_1(X)\in \Gp(A_2)$ by Lemma \ref{lemma property of Gorenstein projective modules} (a).

(b) If $F_1$ is exact or its restriction to $\Gp(A_1)$ is exact, and $\pd F_2(Q)<\infty$ or $\ind F_2(Q)<\infty$ for any indecomposable projective $A_2$-module $Q$, it is easy to see that $F_1$ satisfies the conditions stated in (a) by Lemma \ref{lemma property of Gorenstein projective modules} (b), which yields the result immediately.
\end{proof}

\begin{theorem}\label{theorem recollement of Gorenstein defect categories}
Let $A$ and $B$ be Artin algebras, $_AM_B$ an $A\mbox{-}B$-bimodule such that $M$ as a left $A$-module projective and as a right $B$-module has finite projective dimension, and $\Lambda=\left( \begin{array}{cc} A&M\\0&B \end{array}\right)$.
If $\Hom_A(M,A) \in (\Gproj B)^\bot$, then $D_{def}(\Lambda)$ admits a recollement relative to $D_{def}(A)$ and $D_{def}(B)$. In particular, if $\pd_B\Hom_A(M,P)<\infty$ or $\ind_B\Hom_A(M,P)<\infty$ for every indecomposable projective $A$-module $P$, then $D_{def}(\Lambda)$ admits a recollement relative to $D_{def}(A)$ and $D_{def}(B)$.
\end{theorem}
\begin{proof}
We prove it by applying Theorem \ref{theorem 1} to Theorem \ref{recollement of derived categories}.

Since $_A M$ is projective, it is easy to see that $i^!$ preserves projectives. Lemma \ref{lemma adjoint pair} shows that $i^!$ admits a right adjoint functor $i_?:\mod A\rightarrow \mod \Lambda$ and $i_? (P)\mapsto \left( \begin{array}{cc} P\\ F(P) \end{array} \right)_{\psi_P}$, where $\psi_P=\alpha^{-1}_{P,F(P)}(\Id_{F(P)})$ for any indecomposable projective $A$-module $P$. Obviously, there is a short exact sequence
\begin{equation}\label{equation b}
0\rightarrow \left( \begin{array}{ccc} P\\ 0 \end{array} \right) \rightarrow \left( \begin{array}{ccc} P\\ F(P) \end{array} \right)_{\psi_P}\rightarrow \left( \begin{array}{ccc} 0\\ F(P) \end{array} \right)\rightarrow0.
\end{equation}
For any Gorenstein projective $\Lambda$-module $\left(\begin{array}{cc} X\\ Y   \end{array}  \right)_{\phi}$, in combination with Lemma \ref{lemma zhang}, we know that $\phi$ is monic, $Y$ is Gorenstein projective $B$-module, and $\Coker(\phi)$ is a Gorenstein projective $A$-module. So there is a short exact sequence in $\Gp(\Lambda)$:
\begin{equation}\label{equation a}
0\rightarrow \left(\begin{array}{cc} G(Y)\\ Y   \end{array}  \right)_{\id} \rightarrow \left(\begin{array}{cc} X\\ Y   \end{array}  \right)_{\phi} \rightarrow \left(\begin{array}{cc} \Coker(\phi)\\ 0   \end{array}  \right)\rightarrow0.
\end{equation}
Since $i_*$ is exact and preserves projectives, we get that $\Coker (\phi)$ has the property $(*)$ relative to $i_*$. Then Lemma \ref{lemma extension adjoint functor} yields that
\begin{eqnarray*}\Ext^k_\Lambda( \left(\begin{array}{cc} \Coker(\phi)\\ 0   \end{array}  \right), \left( \begin{array}{ccc} 0\\ F(P) \end{array} \right))&=&\Ext_\Lambda^k(i_*(\Coker (\phi)), \left( \begin{array}{ccc} 0\\ F(P) \end{array} \right) )\\
&\cong&\Ext_A^k(\Coker (\phi), i^!(\left( \begin{array}{ccc} 0\\ F(P) \end{array} \right) ))=0, \mbox{ for any }k>0,
\end{eqnarray*}
since $i^!(\left( \begin{array}{ccc} 0\\ F(P) \end{array} \right) )=0$.
Similarly, from the proof of Proposition \ref{proposition left recollement}, we get that the restriction of $j_!$ to $\Gp(B)$ is exact and preserves projectives, so for any $k>0$, Lemma \ref{lemma extension adjoint functor} shows that
\begin{eqnarray*}\Ext^k_\Lambda( \left(\begin{array}{cc} G(Y)\\ Y   \end{array}  \right)_{\id}, \left( \begin{array}{ccc} 0\\ F(P) \end{array} \right))&=&\Ext_\Lambda^k(j_!(Y), \left( \begin{array}{ccc} 0\\ F(P) \end{array} \right) )\\
&\cong&\Ext_B^k(Y, j^*(\left( \begin{array}{ccc} 0\\ F(P) \end{array} \right) ))\\
&\cong& \Ext_B^k(Y, F(P)),
\end{eqnarray*}
which are all zero since $Y$ is a Gorenstein projetive $B$-module and $F(P)\in (\Gproj B)^\bot$.

By applying $\Hom_\Lambda(-, \left( \begin{array}{ccc} 0\\ F(P) \end{array} \right))$ to the exact sequence (\ref{equation a}), we get that
$$\Ext^k_\Lambda( \left(\begin{array}{cc} X\\ Y   \end{array}  \right)_{\phi}, \left( \begin{array}{ccc} 0\\ F(P) \end{array} \right))=0\mbox{ for any }k>0.$$
Since $\left( \begin{array}{ccc} P\\ 0\end{array} \right)$ is projective and $\left(\begin{array}{cc} X\\ Y   \end{array}  \right)_{\phi}$ is Gorenstein projective, $\Ext^k_\Lambda( \left(\begin{array}{cc} X\\ Y   \end{array}  \right)_{\phi}, \left( \begin{array}{ccc} P\\ 0\end{array} \right))=0$ for any $k>0$. From them, by applying $\Hom_\Lambda( \left(\begin{array}{cc} X\\ Y   \end{array}  \right)_{\phi},  -)$ to the exact sequence (\ref{equation b}), we get that
$$\Ext^k_\Lambda( \left(\begin{array}{cc} X\\ Y   \end{array}  \right)_{\phi}, \left( \begin{array}{ccc} P\\ F(P) \end{array} \right)_{\psi_P})=0\mbox{ for any }k>0.$$
So Lemma \ref{lemma adjoint functor Gorenstein projective} (a) yields that $i^!(\Gp(\Lambda))\subseteq \Gp(A)$.
In combination with $i^!$ is an exact functor, we get that $D^b(i^!)(\langle \Gp(\Lambda)\rangle)\subseteq \langle \Gp(A)\rangle$.

From the recollement in Theorem \ref{recollement of derived categories}, for any $Z\in \Gp(\Lambda)$, there is a triangle in $D^b(\mod\Lambda)$: $i_*i^! Z \rightarrow Z\rightarrow j_*j^* Z\rightarrow i_*i^! Z [1]$. Together with $i_*(\Gp(A))\subseteq \Gp(\Lambda)$, we get that $i_*i^! Z\in \Gp(\Lambda)$ by the above and then $j_*j^*Z\in \langle \Gp(\Lambda)\rangle$. Furthermore, $D^b(j_*)D^b(j^*)(\langle \Gp(\Lambda)\rangle )\subseteq \langle \Gp(\Lambda)\rangle$ since $j^*,j_*$ are exact functors. In combination with Proposition \ref{proposition left recollement}, we get that there is a
recollement of Gorenstein defect categories by applying Theorem \ref{theorem 1} to the recollement in Theorem \ref{recollement of derived categories}:
\[\xymatrix{
D_{def}(A) \ar[r]^{\tilde{i}_*}& D_{def}(\Lambda)
\ar@<2ex>[l]^{\tilde{i}^!}\ar@<-3ex>[l]_{\tilde{i}^*}\ar[r]^{\tilde{j}^*}& D_{def}(B).
\ar@<2ex>[l]^{\tilde{j}_*}\ar@<-3ex>[l]_{\tilde{j}_!}  }\]

For the last statement, if $\pd_B\Hom_A(M,P)<\infty$ or $\ind_B\Hom_A(M,P)<\infty$ for every indecomposable projective $A$-module $P$, then $\Hom_A(M,P)\in (\Gproj B)^\bot$ by Lemma \ref{lemma property of Gorenstein projective modules} (b), which implies that $\Hom_A(M,A)\in (\Gproj B)^\bot$, and the desired result follows immediately.

\end{proof}

\begin{corollary}
Let $A$ and $B$ be Artin algebras, $_AM_B$ a $A\mbox{-}B$-bimodule such that $M$ as a left $A$-module projective and as a right $B$-module has finite projective dimension, and $\Lambda=\left( \begin{array}{cc} A&M\\0&B \end{array}\right)$. If $M\otimes_B Y \in \,^\bot A$ for any Gorenstein projective $B$-module $Y$, then $D_{def}(\Lambda)$ admits a recollement relative to $D_{def}(A)$ and $D_{def}(B)$.
\end{corollary}
\begin{proof}
Since $_AM$ is projective, we get that the functor $G:\mod B\rightarrow \mod A$ preserves projectives.
For any Gorenstein projective $B$-module $Y$ and any projective resolution $\cdots \xrightarrow{d_2} Q_1\xrightarrow{d_1} Q_0 \rightarrow Y\rightarrow0$ of $Y$, we have
$\ker(d_i)\in\Gproj B$, and then $D\mathrm{Tor}^B_1(M,\ker d^i)\cong \Ext_B^1(\ker d^i, DM)\cong0$ by Lemma \ref{lemma property of Gorenstein projective modules} (b), since $\ind_B DM=\pd M_B<\infty$, where $D=\Hom_R(-,E)$ is the duality functor. So $\cdots \xrightarrow{G(d_2)} G(P_1)\xrightarrow{G(d_1)} G(P_0) \rightarrow G(Y)\rightarrow0$ is a projective resolution of
$G(Y)$, and then $Y$ has the property $(*)$ relative to $G$. Together with $(G,F)$ is an adjoint pair, Lemma \ref{lemma extension adjoint functor} yields that
$$ \Ext_B^k(Y, F(A))\cong \Ext_A^k(G(Y), A)=0$$
for any $k>0$ since $G(Y)\in\,^\bot A$. Therefore, $\Hom_A(M, A)\in (\Gproj B)^\bot$, and our desired result follows directly from Theorem \ref{theorem recollement of Gorenstein defect categories}.

\end{proof}

The following corollaries follow from Theorem \ref{theorem recollement of Gorenstein defect categories} immediately.
\begin{corollary}
Let $A$ and $B$ be Artin algebras, $_AM_B$ a $A\mbox{-}B$-bimodule such that $M$ as a left $A$-module projective and as a right $B$-module has finite projective dimension, and $\Lambda=\left( \begin{array}{cc} A&M\\0&B \end{array}\right)$. If $\pd_B \Hom_A(M,A)<\infty$, or $\ind_B\Hom_A(M,A)<\infty$, then $D_{def}(\Lambda)$ admits a recollement relative to $D_{def}(A)$ and $D_{def}(B)$.
\end{corollary}

\begin{corollary}\label{corollary special triangle matrix algebra}
Let $A$ be an Artin algebra. Then $D_{def}(\left( \begin{array}{cc} A&A\\0&A \end{array}\right))$ admits a recollement relative to $D_{def}(A)$ and $D_{def}(A)$.
\end{corollary}

\begin{example}\label{example 2}
If $_AM$ is not projective, then the recollement of Gorenstein defect categories obtained in the proof of Theorem \ref{theorem recollement of Gorenstein defect categories} does not exist in general.

Let $K$ be a field and $Q$ be the following quiver
\[\xymatrix{ 4 \ar@/^/[r]^{\alpha_1}  \ar[d]_{\gamma_1} & 5 \ar@/^/[r]^{\alpha_2}\ar@/^/[l]^{\beta_1}  \ar[d]^{\gamma_2}& 6 \ar@/^/[l]^{\beta_2} & 7\ar[l]^{\beta_3}  \\
1  \ar@/^/[r]^{\theta_1} & 2 \ar[r]^{\theta_2} \ar@/^/[l]^{\delta_1} & 3 . }\]
Let $\Lambda=KQ/I$ where $I$ is the ideal of $KQ$ generated by $\alpha_1\beta_1$, $\beta_1\alpha_1$, $\alpha_2\beta_2$, $\beta_2\alpha_2$, $\beta_2\beta_3$, $\theta_1\delta_1$, $\delta_1\theta_1$, $\theta_2\theta_1$, $\theta_2\gamma_2$, $\gamma_2\alpha_1-\theta_1\gamma_1$ and $\delta_1 \gamma_2-\gamma_1\beta_1$.
Denote by $e_i$ the idempotent corresponding to the vertex $i$ for $1\leq i\leq 7$. Set $e=e_1+e_2+e_3$, $A=e\Lambda e$ and $B=(1-e)\Lambda(1-e)$.
Then $\Lambda= \left( \begin{array}{cc} A&M\\0&B \end{array}\right)$. Obviously, $\pd_A M=1$ and $M_B$ is projective.

Let $S_i$ be the simple $\Lambda$-module corresponding to $i$. We can view $S_i$ as the simple $A$-module (resp. $B$-module) for $1\leq i\leq 3$ (resp. $4\leq i\leq 7$).
From \cite{chen3} or the Auslander-Reiten quiver of $B$, we get that the indecomposable Gorenstein projective $B$-modules are precisely $S_4$ and the string module $N$ corresponding to the string $5\xrightarrow{\alpha_2} 6$. $j_!(S_4)$ (resp. $j_!(N)$) is the string module corresponding to the string $4\xrightarrow{\gamma_1} 1$ (resp. $2\xleftarrow{\gamma_2} 5\xrightarrow{\alpha_2}6$).

Suppose for a contradiction that the recollement of Gorenstein defect categories obtained in the proof of Theorem \ref{theorem recollement of Gorenstein defect categories} exists. Then both $i^! j_!(S_4)= S_1$ and $i^! j_!(N)= S_2$ are in $\langle \Gp(A)\rangle$. Since $S_3$ is projective, all the simple $A$-modules are in $\langle \Gp(A)\rangle$.
Since the thick subcategory of $D^b(A)$ consisting of all the simple $A$-modules is $D^b(A)$, we get that $D^b(A)=\langle \Gp(A)\rangle$, which implies that $A$ is Gorenstein. However, it is easy to check that $A$ is not Gorenstein, which yields a contradiction.
\end{example}

\begin{example}
Even though $_AM$ and $M_B$ are projective, if $\Hom_A(M,A) \notin (\Gproj B)^\bot$, or in particular, neither $\pd_B\Hom_A(M,P)<\infty$ nor $\ind_B\Hom_A(M,P)<\infty$ for some indecomposable projective $A$-module $P$, then the recollement of Gorenstein defect categories obtained in the proof of Theorem \ref{theorem recollement of Gorenstein defect categories} does not exist in general.

Let $K$ be a field and $Q$ be the quiver as in Example \ref{example 2}.
Let $\Lambda=KQ/I$ where $I$ is the ideal of $KQ$ generated by $\alpha_1\beta_1$, $\beta_1\alpha_1$, $\alpha_2\beta_2$, $\beta_2\alpha_2$, $\beta_2\beta_3$, $\theta_1\delta_1$, $\delta_1\theta_1$, $\theta_2\theta_1$, $\gamma_2\alpha_1-\theta_1\gamma_1$ and $\delta_1 \gamma_2-\gamma_1\beta_1$.
Denote by $e_i$ the idempotent corresponding to the vertex $i$ for $1\leq i\leq 7$. Set $e=e_1+e_2+e_3$, $A=e\Lambda e$ and $B=(1-e)\Lambda(1-e)$.
Then $\Lambda= \left( \begin{array}{cc} A&M\\0&B \end{array}\right)$. Obviously, $_A M$ and $M_B$ are projective.

Keep the notations as in Example \ref{example 2}.
Then the indecomposable Gorenstein projective $B$-modules are precisely $S_4$ and the string module $N$. $j_!(S_4)$ (resp. $j_!(N)$) is the string module corresponding to the string $4\xrightarrow{\gamma_1} 1$ (resp. $3\xleftarrow{\theta_2}2\xleftarrow{\gamma_2} 5\xrightarrow{\alpha_2}6$).

Easily, $A$ is representation-finite, and from its Auslander-Reiten quiver, we get that
$$\Ext_B^2(S_4,\Hom_A(M,\,_AP_2))= \Ext^2_A(M\otimes_B S_4,\, _AP_2 )= \Ext^2_A(_AS_1,\,_AP_2)\neq0,$$
where $_AP_2$ is the indecomposable projective $A$-module corresponding to the vertex $2$, and $_A S_1$ is the simple $A$-module corresponding to the vertex $1$.
So $\Hom_A(M,A) \notin (\Gproj B)^\bot$.

Suppose for a contradiction that the recollement of Gorenstein defect categories obtained in the proof of Theorem \ref{theorem recollement of Gorenstein defect categories} exists. Then $i^! j_!(S_4)= \,_AS_1$ and $i^! j_!(N)$ satisfy that both of them are in $\langle \Gp(A)\rangle$, where $i^! j_!(N)$ is the string module with its string $3\xleftarrow{\theta_2}2$. Since $S_3$ is projective, similar to Example \ref{example 2}, it is easy to see that all the simple $A$-modules are in $\langle \Gp(A)\rangle$, which implies that $A$ is Gorenstein. However, $A$ is not Gorenstein, which yields a contradiction.
\end{example}

\section{Application: simple gluing algebras}
In this section, let us consider a special kind of triangular matrix algebras. We always assume that $K$ is a field.
Let $A=KQ_A/I_A$ and $B=KQ_B/I_B$ be two finite-dimensional bound quiver algebras. Fix vertices $v_1,\dots,v_m\in Q_A$, $w_1,\dots,w_n\in Q_B$ for some integers $m,n>0$. We define a new quiver $Q$ from $Q_A$ and $Q_B$ by adding $a_{ij}$ ($a_{ij}\geq0$) arrows from $w_j$ to $v_i$ for any $1\leq i\leq m$, $1\leq j\leq n$. In this way, we can view $Q_A$ and $Q_B$ as the full subquivers of $Q$, and call $Q$ a \emph{simple gluing quiver} of $Q_A$ and $Q_B$. Let $I=\langle I_A,I_B\rangle$, and $\Lambda=KQ/I$. Then $\Lambda$ is finite-dimensional, and we call $\Lambda$ a \emph{simple gluing algebra} of $A$ and $B$.

\setlength{\unitlength}{0.5mm}
\begin{center}
\begin{picture}(100,50)(0,-20)
\qbezier(-5,0)(0,30)(40,10)
\qbezier(-5,0)(0,-30)(40,-10)
\qbezier(40,10)(55,0)(40,-10)
\put(40,0){\circle*{2}}

\qbezier(70,10)(55,0)(70,-10)
\qbezier(70,10)(110,30)(115,0)
\qbezier(70,-10)(110,-30)(115,0)
\put(70,0){\circle*{2}}
\put(40,0){\vector(1,0){29}}

\put(33,-5){$w_1$}
\put(72,-5){$v_1$}
\put(50,2){\small$a_{11}$}
\put(10,-5){$Q_B$}
\put(90,-5){$Q_A$}

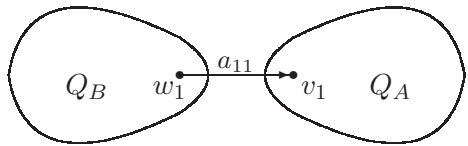
\put(-30,-31){Figure 1. Simple gluing quiver of $Q_A$ and $Q_B$}
\put(1,-40){  for the case $m=1=n$.}
\end{picture}
\vspace{1cm}
\end{center}

For a simple gluing algebra $\Lambda$ of $A$ and $B$, obviously, $\Lambda=\left( \begin{array}{cc} A&M\\0&B \end{array}\right)$ with $_A M$ and $M_B$ projective.
So we get the Gorenstein property of $\Lambda$ by the following general result.

\begin{lemma}[\cite{Chen1}]\label{Gorenstein dimension of upper triangular matrix algebra}
Let $A$ and $B$ be two Gorenstein algebras, $_AM_B$ an $A\mbox{-}B$-bimodule such that $_AM$ and $M_B$ are projective, and $\Lambda=\left( \begin{array}{cc} A&M\\0&B \end{array}\right)$. Then $\Lambda$ is Gorenstein. In particular, if $\Gd A\neq\Gd B$, then $\Gd (\Lambda)=\max\{\Gd A,\Gd B\}$;
otherwise, $\Gd ( \Lambda)\leq \Gd A+1$.
\end{lemma}
\begin{proof}
The proof of this lemma is similar to \cite[Theorem 3.3]{Chen1}, we omit it here.
\end{proof}

By \cite[Theorem 2.5]{LL}, we get the following recollement:
\[\xymatrix{
D_{sg}(A) \ar[r]^{\tilde{i}_*}& D_{sg}(\Lambda)
\ar@<2ex>[l]^{\tilde{i}^!}\ar@<-3ex>[l]_{\tilde{i}^*}\ar[r]^{\tilde{j}^*}& D_{sg}(B)\quad (*)
\ar@<2ex>[l]^{\tilde{j}_*}\ar@<-3ex>[l]_{\tilde{j}_!}  }\]
where $\tilde{i}^*$, $\tilde{i}_*$, $\tilde{i}^!$, $\tilde{j}_!$, $\tilde{j}^*$ and $\tilde{j}_*$ are induced by the six structure functors in Theorem \ref{recollement of derived categories}.

For simple gluing algebra $\Lambda$, it is worth noting that all the six functors $i^*,i_*,i^!,j_!,j^*,j_*$ are exact functors. In fact, $i^*= A\otimes_\Lambda-$, together with that $A$ is projective as $\Lambda$-module, it is easy to see that $i^*$ is exact; Since $M_B$ is projective, we get that $G=M\otimes_B-$ is exact, from the definition of $j_!$, it is easy to see that $j_!$ is exact.

\begin{proposition}\label{lemma recollement splits}
Let $\Lambda$ be a simple gluing algebra of $A$ and $B$. Then
$$D_{sg}(\Lambda)\simeq D_{sg}(A)\coprod D_{sg}(B).$$
\end{proposition}
\begin{proof}
For any $\Lambda$-module $Z$, by the recollement $(*)$, there exists a triangle in $D_{sg}(\Lambda)$:
\begin{equation}\label{equation triangle}
\tilde{j}_! \tilde{j}^*Z\rightarrow Z\rightarrow\tilde{i}_*\tilde{i}^* Z\rightarrow \tilde{j}_! \tilde{j}^*Z[1],
\end{equation}
where $[1]$ is the suspension functor.

First, for any $Y\in \mod(B)$, let $(Y_i,Y_\alpha)$ be the corresponding representation of $B$. It is easy to see that $$M\otimes_B Y=\bigoplus_{j=1}^n (\bigoplus_{i=1}^m P_{v_i}^{\oplus a_{ij}\dim Y_{w_j}})\in\proj A,$$ 
where $P_{v_i}$ is the indecomposable projective $A$-module corresponding to the vertex $v_i$.
So
$$i^! j_!(Y)= i^!( \left( \begin{array}{cc} M\otimes_B Y\\Y \end{array} \right))=M\otimes_B Y,$$
which is projective.
Then $\tilde{i}^! \tilde{j}_!=0$ since $\mod B$ is a generator of $D^b(B)$ and $i^!,j_!$ are exact functors.

For (\ref{equation triangle}), by the recollement $(*)$, we know that $(\tilde{i}_*,\tilde{i}^!)$ is an adjoint pair, and then
$$\Hom_{D_{sg}(\Lambda)}(\tilde{i}_*\tilde{i}^* Z, \tilde{j}_! \tilde{j}^*Z[1] )\cong \Hom_{D_{sg}(A)}( \tilde{i}^* Z, \tilde{i}^!\tilde{j}_! \tilde{j}^*Z[1])=0,$$
which implies that $Z\cong \tilde{j}_! \tilde{j}^*Z\oplus \tilde{i}_*\tilde{i}^* Z$. In particular, from the above, we also get that
$$\Hom_{D_{sg}(\Lambda)}( \Im \tilde{i}_*,\Im \tilde{j}_!)=0.$$
On the other hand, it is easy to see that $\Hom_{D_{sg}(\Lambda)}(\Im \tilde{j}_!, \Im \tilde{i}_*)=0$ by the definition of recollement. So $D_{sg}(\Lambda)\simeq \Im \tilde{i}_* \coprod \Im \tilde{j}_!$ and then
$$D_{sg}(\Lambda)\simeq D_{sg}(A)\coprod D_{sg}(B),$$
since $\tilde{i}_*,\tilde{j}_!$ are full embeddings.
\end{proof}

Note that we can view $\underline{\Gproj}(\Lambda)$ as full subcategory of $D_{sg}(\Lambda)$ by Buchweitz's Theorem for any Artin algebra $\Lambda$, parallel to the above, we get the following lemma.

\begin{lemma}\label{lemma Gp split}
Let $\Lambda$ be a simple gluing algebra of $A$ and $B$. Then
$$\underline{\Gproj}(\Lambda)\simeq \underline{\Gproj}(A)\coprod\underline{\Gproj}(B).$$
Furthermore, for any indecomposable $\Lambda$-module $Z$, $Z$ is Gorenstein projective if and only if there exists an indecomposable Gorenstein projective $A$-module $X$ or $B$-module $Y$ such that $Z\cong i_*(X)$ or $Z\cong j_!(Y)$.
\end{lemma}
\begin{proof}
From the proof of Proposition \ref{proposition left recollement}, we get that $i^*,i_*,j_!,j^*$ preserve Gorenstein projective modules, so these functors induce triangulated functors on the stable categories of Gorenstein projective modules, which we also denote by $\tilde{i}^*,\tilde{i}_*,\tilde{j}_!,\tilde{j}^*$. For any Gorenstein projective $\Lambda$-module $\left( \begin{array}{cc} X\\Y \end{array} \right)_\phi$, we have that $\phi:M\otimes_BY\rightarrow X$ is monic, $\Coker{\phi}\in \Gproj A,Y\in \Gproj B$. So
there exists a triangle in $\underline{\Gproj}(\Lambda)$
\begin{equation}\label{equation triangle 2}
\left( \begin{array}{cc} M\otimes Y\\Y \end{array} \right)_{\Id} \rightarrow  \left( \begin{array}{cc} X\\Y \end{array} \right)_\phi \rightarrow \left( \begin{array}{cc} \Coker \phi\\0 \end{array} \right)_0\xrightarrow{f} \left( \begin{array}{cc} M\otimes Z\\Z \end{array} \right)_{\Id},
\end{equation}
where $Z$ is a Gorenstein projective $B$-module such that there is an exact sequence
$0\rightarrow Y \rightarrow Q_Y \rightarrow Z\rightarrow0$ with $Q_Y$ a projective $B$-module and $Z\in \Gproj(B)$.
Similar to the proof of Proposition \ref{lemma recollement splits}, we get that $f=0$ in $\underline{\Gproj}(\Lambda)$, and so
(\ref{equation triangle 2}) splits, which implies that
$$ \left( \begin{array}{cc} X\\Y \end{array} \right)_\phi\cong  \tilde{j}_! \tilde{j}^* (\left( \begin{array}{cc} X\\Y \end{array} \right)_\phi)\oplus \tilde{i}_*\tilde{i}^* (\left( \begin{array}{cc} X\\Y \end{array} \right)_\phi).$$

Furthermore, similar to the proof of Proposition \ref{lemma recollement splits}, we get that
$$\Hom_{\underline{\Gproj}(\Lambda)}( \Im \tilde{i}_*,\Im \tilde{j}_!)=0= \Hom_{\underline{\Gproj}(\Lambda)}(\Im \tilde{j}_!, \Im \tilde{i}_*),$$
so $\underline{\Gproj}(\Lambda)\simeq \Im \tilde{i}_* \coprod \Im \tilde{j}_!$ and then
$$\underline{\Gproj}(\Lambda)\simeq \underline{\Gproj}(A)\coprod \underline{\Gproj}(B),$$
since $\tilde{i}_*,\tilde{j}_!$ are full embeddings.

For the last statement, first, for any indecomposable projective $\Lambda$-module $U_i$ corresponding to the vertex $i$, if $i\in Q_A$, then it is easy to see that $i_*(P_i)=U_i$, where $P_i$ is the indecomposable projective $A$-module corresponding to the vertex $i$; otherwise, $j_!(Q_i)=U_i$, where $Q_i$ is the indecomposable projective $B$-module corresponding to the vertex $i$.

Second, for any indecomposable non-projective Gorenstein projective $\Lambda$-module $Z$, from $\underline{\Gproj}(\Lambda)\simeq \Im \tilde{i}_* \coprod \Im \tilde{j}_!$, it is either in $\Im \tilde{i}_* $ or in $\Im \tilde{j}_!$. Since $i_*$ and $j_!$ preserve projective modules and are full embeddings by Lemma \ref{recollement of abelian categories}, it is routine to check that there exists an indecomposable Gorenstein projective $A$-module $X$ or $B$-module $Y$ such that $Z\cong i_*(X)$ or $Z\cong j_!(Y)$.
\end{proof}

Now we get the main result of this section.
\begin{theorem}\label{theorem defect Gorenstein split}
Let $\Lambda$ be a simple gluing algebra of $A$ and $B$. Then
 $$D_{def}(\Lambda)\simeq D_{def}(A)\coprod D_{def}(B).$$
\end{theorem}
\begin{proof}
From Proposition \ref{lemma recollement splits} and Lemma \ref{lemma Gp split}, we get that
$$D_{sg}(\Lambda)\simeq D_{sg}(A)\coprod D_{sg}(B),\mbox{ and }\underline{\Gproj}(\Lambda)\simeq \underline{\Gproj}(A)\coprod\underline{\Gproj}(B).$$
In particular, both of the two equivalences are induced by the same functors $i_*,j_!$, so they are compatible, which implies the result by the definition of Gorenstein defect category immediately.
\end{proof}

\begin{example}\label{example 3}
For triangular matrix algebra $\Lambda= \left( \begin{array}{cc} A&M\\0&B \end{array}\right)$. Even though $M$ is projective both as left $A$-module and as right $B$-module, in general,
we do not have 
$$D_{sg}(\Lambda)\simeq D_{sg}(A)\coprod D_{sg}(B), \mbox{ or } \underline{\Gproj}(\Lambda)\simeq \underline{\Gproj}(A)\coprod\underline{\Gproj}(B).$$

Let $\Lambda=KQ_\Lambda/I_\Lambda$, where $Q_\Lambda$ is the quiver as Figure 2 shows, and $I_\Lambda=\langle \varepsilon_1^2,\varepsilon_2^2, \varepsilon_1\alpha-\alpha\varepsilon_2\rangle$. Let $A=e_1 \Lambda e_1$ and $B=e_2\Lambda e_2$ where $e_i$ is the idempotent corresponding to the vertex $i$ for $i=1,2$. Then
$\Lambda$ is of form $\left( \begin{array}{cc} A&M\\0&B \end{array}\right)$ with $M$ projective both as left $A$-module and as right $B$-module. Then all these algebras are Gorenstein algebras.
It is easy to see that $D_{sg}(A)\simeq D_{sg}(B)=\underline{\mod}(K[X]/\langle X^2\rangle)$, and $D_{sg}(\Lambda)\simeq D^b(KQ)/ [1]$, where $Q$ is a quiver of type $\A_2$, and $[1]$ is the suspension functor, see \cite{RZ}. However, $D_{sg}(\Lambda)$ is not equivalent to $D_{sg}(A)\coprod D_{sg}(B)$.

\begin{center}\setlength{\unitlength}{0.7mm}
 \begin{picture}(40,10)(0,10)

\put(0,-2){$1$}

\put(18,0){\vector(-1,0){14}}
\put(10,-4){$\alpha$}
\put(20,-2){$2$}

\qbezier(-1,1)(-3,3)(-2,5.5)
\qbezier(-2,5.5)(1,9)(4,5.5)
\qbezier(4,5.5)(5,3)(3,1)
\put(3.1,1.4){\vector(-1,-1){0.3}}

\qbezier(19,1)(17,3)(18,5.5)
\qbezier(18,5.5)(21,9)(24,5.5)
\qbezier(24,5.5)(25,3)(23,1)
\put(23.1,1.4){\vector(-1,-1){0.3}}

\put(1,10){$\varepsilon_1$}
\put(21,10){$\varepsilon_2$}

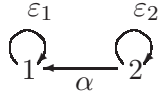
\put(-40,-13){Figure 2. The quiver $Q_\Lambda$ in Example \ref{example 3}. }
\end{picture}

\vspace{1.9cm}
\end{center}
\end{example}

\end{document}